\documentclass[12pt]{amsart}
\usepackage{txfonts}

\usepackage{amssymb}
\usepackage{mathrsfs}
\usepackage{amsmath}
\usepackage{amssymb,latexsym}

\newtheorem{thm}{Theorem}[section]
\newtheorem{cor}[thm]{Corollary}

\theoremstyle{definition}
\newtheorem{defn}[thm]{Definition}

\numberwithin{equation}{section}

\begin{document}
\baselineskip=17pt

\title [error bounds for quasi-Monte Carlo integration  for $\mathscr{L}_{\infty}$]{error bounds for quasi-Monte Carlo integration  for $\mathscr{L}_{\infty}$ with uniform point sets}

\author[Su Hu]{Su Hu}
\address{Department of Mathematical Sciences\\ Tsinghua University\\
Beijing 100084, China} \email{hus04@mails.tsinghua.edu.cn}

\author[Yan Li]{Yan Li}
\address{Department of Mathematical Sciences\\ Tsinghua University\\
Beijing 100084, China}
\email{liyan\_00@mails.tsinghua.edu.cn}

\begin{abstract} Niederreiter~\cite{Nie} established new bounds for quasi-Monte Carlo integration for nodes sets with a special kind of uniformity property. Let $(X,\mathscr{A},\mu)$ be an arbitrary probability space, i.e., $X$ is an arbitrary nonempty set, $\mathscr{A}$ a $\sigma$-algebra of subsets of $X$, and $\mu$ a probability measure defined on $\mathscr{A}$. The functions considered in ~\cite{Nie} are bounded $\mu$-integrable functions on $X$. In this note, we extend some of his results for  bounded  $\mu$-integrable functions to essentially bounded $\mathscr{A}$-measurable functions. So Niederreiter's bounds can be used in a more general setting.
\end{abstract}

 \maketitle

\section{Introduction}
  Let $(X,\mathscr{A},\mu)$ be an arbitrary probability space, i.e., $X$ is an arbitrary nonempty set, $\mathscr{A}$ a $\sigma$-algebra of subsets of $X$, and $\mu$ a probability measure defined on $\mathscr{A}$. Niederreiter~\cite{Nie} established new bounds for quasi-Monte Carlo integration for nodes sets with a special kind of uniformity property. The functions considered in ~\cite{Nie} are bounded $\mu$-integrable functions on $X$. In this note, we extend some of his results for bounded  $\mu$-integrable functions to essentially bounded $\mathscr{A}$-measurable functions.

\section{Main results}
Let $\mathscr{L}_{\infty}(X,\mathscr{A},\mu)$ be the set of all essentially bounded $\mathscr{A}$-measurable functions on $X$, two functions being identified if they differ only on a $\mu$-null set.
For any $\mathscr{A}$-measurable function $g$ on $X$,  $||g||_{\infty} (\textrm{esssup}_{x\in X}|g|)$  denotes  the essential supremum of $|g|$ (see P.346 of ~\cite{Ed}). For an extended real-valued function $f$, we define
$f^{+}=\textrm{max}\{f,0\} ~\textrm{and}~ f^{-}=-\textrm{min}\{f,0\}.$
Notice that $f^{+}\geq 0,  f^{-}\geq 0$, and $f=f^{+}-f^{-}$ (see P.164 of ~\cite{Ed}). For a given nonempty subset $\mathscr{M}$ of $\mathscr{A}$, let $L_{\mathscr{M}}$ be linear subspace of $\mathscr{L}_{\infty}(X,\mathscr{A},\mu)$ spanned by the constant function 1 and all characteristic functions $\chi_{M}$, $M\in\mathscr{M}$. For any $f\in\mathscr{L}_{\infty}(X,\mathscr{A},\mu)$, let $D(f,L_{\mathscr{M}})$ be the $\mathscr{L}_{\infty}$ distance from $f$ to $\mathscr{L}_{\mathscr{M}}$, that is ,$$D(f,L_{\mathscr{M}})=\textrm{inf}_{l\in\mathscr{L}_{\mathscr{M}}}||f-l||_{\infty}.$$
The following definition can be found in P.285 of ~\cite{Nie}.
\begin{defn}\label{UN} Let $(X,\mathscr{A},\mu)$ be an arbitrary probability space, let $\mathscr{M}$ be a nonempty subset of $\mathscr{A}$. A point set $\mathscr{P}$ of $N$ elements of $X$ is called $(\mathscr{M},\mu)$-uniform if
$$\sum_{i=1}^{N}\chi_{M}(X_{n})=A(M;\mathscr{P})=\mu(M)N, \textrm{for all}~ M\in\mathscr{M}.$$
\end{defn}
Let $(\underbrace{X\times...\times X}_{N},\underbrace{\mathscr{A}\times...\times\mathscr{A}}_{N},\underbrace{\mu\times...\times\mu}_{N})$  be the product measurable space (see P.379 of~\cite{Ed}). We can view a point set $\mathscr{P}=\{X_{1},...,X_{N}\}$ as a point in $\underbrace{X\times...\times X}_{N}$ and
$\frac{1}{N}\sum_{n=1}^{N}f(X_{n})$ as a $N$-variable function on $\underbrace{X\times...\times X}_{N}$. Since  $f\in\mathscr{L}_{\infty}(X,\mathscr{A},\mu)$, we have $\frac{1}{N}\sum_{n=1}^{N}f(X_{n})\in\mathscr{L}_{\infty}(\underbrace{X\times...\times X}_{N},\underbrace{\mathscr{A}\times...\times\mathscr{A}}_{N},\underbrace{\mu\times...\times\mu}_{N})$.

Let $$\mathscr{C}=\{(X_{1},...,X_{N})\in\underbrace{X\times...\times X}_{N}|\mathscr{P}=\{X_{1},...,X_{N}\} \textrm{is an} (\mathscr{M},\mu)-\textrm{uniform point set}\}.$$ Since $f(X_{1},X_{2}...,X_{N})=\sum_{i=1}^{N}\chi_{M}(X_{n})$ is a measurable function on $X\times...\times X$, from Definition~\ref{UN} and Lemma 11.9 of ~\cite{Ed}, if $\mathscr{M}$ is a countable nonempty subset of $\mathscr{A}$, then $\mathscr{C}$ is a measurable set.
\begin{thm}\label{main} Let $(X,\mathscr{A},\mu)$ be an arbitrary probability space. Let $\mathscr{M}$ be a countable nonempty subset of $\mathscr{A}$. Let $$\mathscr{C}=\{(X_{1},...,X_{N})\in\underbrace{X\times...\times X}_{N}|\mathscr{P}=\{X_{1},...,X_{N}\} \textrm{is an} (\mathscr{M},\mu)-\textrm{uniform point set}\}.$$ Then for any $f\in\mathscr{L}_{\infty}(X,\mathscr{A},\mu)$, we have  $$\textrm{esssup}_{(X_{1},...,X_{N})\in\mathscr{C}}|\frac{1}{N}\sum_{n=1}^{N}f(X_{n})-\int_{X}f d\mu|\leq 2D(f,L_{\mathscr{M}}).$$
\end{thm}
\begin{proof}
We extend Niederreiter's proof for Theorem 1 of ~\cite{Nie} to our case.
For any $M\in\mathscr{M}$ and any $(\mathscr{M},\mu)-$uniform point set $\mathscr{P}=\{X_{1},...,X_{N}\}$. we have $$\frac{1}{N}\sum_{n=1}^{N}\chi_{M}(X_{n})=\frac{A(M;\mathscr{P})}{N}=\mu(M)=\int_{X}\chi_{M} d\mu $$ by the definition of an $(\mathscr{M},\mu)-$uniform point set.
\\For any $l\in\mathscr{L}_{\mathscr{M}}$ and any $(\mathscr{M},\mu)-$uniform point set $\mathscr{P}=\{X_{1},...,X_{N}\}$, we have
$$\int_{X}l d\mu=\frac{1}{N}\sum_{n=1}^{N}l(X_{n}).$$  Thus for any $f\in\mathscr{L}_{\infty}(X,\mathscr{A},\mu)$ and any $(\mathscr{M},\mu)-$uniform point set $\mathscr{P}=\{X_{1},...,X_{N}\}$ we have
 \begin{equation*}\begin{aligned}&\frac{1}{N}\sum_{n=1}^{N}f(X_{n})-\int_{X}f d\mu \\&=\frac{1}{N}\sum_{n=1}^{N}(f-l)(X_{n})+\frac{1}{N}\sum_{n=1}^{N}l(X_{n})-\int_{X}(f-l) d\mu-\int_{X}l d\mu
 \\&=\frac{1}{N}\sum_{n=1}^{N}(f-l)(X_{n})-\int_{X}(f-l) d\mu\end{aligned}\end{equation*} for all $l\in L_{\mathscr{M}}$.
\\So
\begin{equation*}\begin{aligned}&\textrm{esssup}_{(X_{1},...,X_{N})\in\mathscr{C}}|\frac{1}{N}\sum_{n=1}^{N}f(X_{n})-\int_{X}f d\mu|
 \\&\leq \textrm{esssup}_{(X_{1},...,X_{N})\in\mathscr{C}}|\frac{1}{N}\sum_{n=1}^{N}(f-l)(X_{n})|+\int_{X}|f-l| d\mu\\&\leq \textrm{esssup}_{(X_{1},...,X_{N})\in\mathscr{C}}|\frac{1}{N}\sum_{n=1 }^{N}(f-l)(X_{n})|+\int_{\{x\in X||f-l|>||f-l||_{\infty}\}}|f-l| d\mu\\&+\int_{\{x\in X||f-l|\leq||f-l||_{\infty}\}}|f-l|d\mu \end{aligned}\end{equation*} for all $l\in L_{\mathscr{M}}$.
\\Since
\begin{equation*}\begin{aligned}&\{(X_{1},...,X_{N})\in\mathscr{C}\mid|\frac{1}{N}\sum_{n=1}^{N}(f-l)(X_{n})| > ||f-l||_{\infty}\}\\&\subset\{(X_{1},...,X_{N})\in\underbrace{X\times...\times X}_{N}\mid|\frac{1}{N}\sum_{n=1}^{N}(f-l)(X_{n})|> ||f-l||_{\infty}\}
 \\&\subset\cup_{n=1}^{N} (X\times...\times\{X_{n}\in X\mid|(f-l)(X_{n})|>||f-l||_{\infty}\}...\times X),\end{aligned}\end{equation*}
we have
\begin{equation*}\begin{aligned}&\underbrace{\mu\times...\times\mu}_{N}(\{(X_{1},...,X_{N})\in\mathscr{C}\mid|\frac{1}{N}\sum_{n=1}^{N}(f-l)(X_{n})| > ||f-l||_{\infty}\})\\&\leq\underbrace{\mu\times...\times\mu}_{N}(\{(X_{1},...,X_{N})\in\underbrace{X\times...\times X}_{N} \mid|\frac{1}{N}\sum_{n=1}^{N}(f-l)(X_{n})| > ||f-l||_{\infty}\})\\&\leq\sum_{n=1}^{N}\underbrace{\mu\times...\times\mu}_{N}(X\times...\times\{X_{n}\in X\mid|(f-l)(X_{n})|>||f-l||_{\infty}\}...\times X)
 \\&=\sum_{n=1}^{N}\mu(\{X_{n}\in X\mid|(f-l)(X_{n})|>||f-l||_{\infty}\})=0 ,\end{aligned}\end{equation*} the last equality follows from Fubini's theorem (see P.384 of~\cite{Ed}).\\From the definition of $\mathscr{L}_{\infty}$-norm, we have $$\textrm{esssup}_{(X_{1},...,X_{N})\in\mathscr{C}}|\frac{1}{N}\sum_{n=1}^{N}(f-l)(X_{n})|\leq ||f-l||_{\infty}.$$
Also from the definition of $\mathscr{L}_{\infty}$-norm, we have $$\mu(\{x\in X\mid|f-l|>||f-l||_{\infty}\})=0,$$ thus $$\int_{\{x\in X\mid|f-l|>||f-l||_{\infty}\}}|f-l| d\mu=0.$$ So \begin{equation*}\begin{aligned}&\textrm{esssup}_{(X_{1},...,X_{N})\in\mathscr{C}}|\frac{1}{N}\sum_{n=1}^{N}f(X_{n})-\int_{X}f d\mu|
  \\&\leq ||f-l||_{\infty}+\int_{\{x\in X\mid|f-l|\leq||f-l||_{\infty}\}}|f-l|d\mu\\&\leq 2||f-l||_{\infty} \end{aligned}\end{equation*} for all $l\in L_{\mathscr{M}}$.
\end{proof}
Let $\mathscr{M}=\{M_{1},...,M_{k}\}$ be a finite nonempty subset of $\mathscr{A}$ such that  $M_{1},...,M_{k}$ are disjoint and $\cup_{i=1}^{k}{M_{i}}=X$. If $f\in\mathscr{L}_{\infty}(X,\mathscr{A},\mu)$, then
$f\in\mathscr{L}_{\infty}(X,\mathscr{A}|_{M_{i}},\mu|_{M_{i}})$ for any $1\leq i\leq k$.
Let
\begin{equation*}{G_{j}(f)=}\begin{cases}-\textrm{inf}_{x\in M_{j}} f^{-}, &\textrm{if}\  \mu(\{x\in M_{j}|f^{+}(x)>0\})=0;\\
\textrm{esssup}_{x\in M_{j}} f^{+}, &\textrm{otherwise},
        \end{cases}
\end{equation*}
\begin{equation*}{g_{j}(f)=}\begin{cases}\textrm{inf}_{x\in M_{j}} f^{+}, &\textrm{if}\ \mu(\{x\in M_{j}|f^{-}(x)>0\})=0;\\
-\textrm{esssup}_{x\in M_{j}} f^{-}, &\textrm{otherwise},
        \end{cases}
\end{equation*}
for $1\leq j\leq k$.
Define $$S_{\mathscr{M}}(f)=\textrm{max}_{1\leq j\leq k}(G_{j}(f)-g_{j}(f)).$$
\begin{cor}Let $(X,\mathscr{A},\mu)$ be an arbitrary probability space. Let $\mathscr{M}=\{M_{1},...,M_{k}\}$ be a finite nonempty subset of $\mathscr{A}$ such that  $M_{1},...,M_{k}$ are disjoint  and $\cup_{j=1}^{k}{M_{j}}=X$. Let $$\mathscr{C}=\{(X_{1},...,X_{N})\in\underbrace{X\times...\times X}_{N}|\mathscr{P}=\{X_{1},...,X_{N}\} \textrm{is an} (\mathscr{M},\mu)-\textrm{uniform point set}\}.$$ Then for any $f\in\mathscr{L}_{\infty}(X,\mathscr{A},\mu)$, we have  $$\textrm{esssup}_{(X_{1},...,X_{N})\in\mathscr{C}}|\frac{1}{N}\sum_{n=1}^{N}f(X_{n})-\int_{X}f d\mu|\leq S_{\mathscr{M}}(f).$$
\end{cor}
\begin{proof} We extend Niederreiter's proof for Corollary  1 of ~\cite{Nie} to our case.
Let $$C_{j}=\frac{1}{2}(G_{j}(f)+g_{j}(f))$$ for $1 \le j \le k$, let $$l=\sum_{j=1}^{k}C_{j}\chi_{M_{j}}.$$
Since $$\{t\in M_{j}\mid|f(t)-C_{j}|>\frac{G_{j}(f)-g_{j}(f)}{2}\}\subset \{t\in M_{j}| f(t)>G_{j}(f)\}\cup\{t\in M_{j}\mid f(t)<g_{j}(f)\},$$
we have $$\mu(\{t\in M_{j}\mid|f(t)-C_{j}|>\frac{G_{j}(f)-g_{j}(f)}{2}\})\leq\mu(\{t\in M_{j}\mid f(t) >G_{j}(f)\})+\mu(\{t\in M_{j}\mid f(t) <g_{j}(f)\})=0$$ by the definition of $\mathscr{L}_{\infty}$-norm.
\\For $t\in M_{j}$,
we have  \begin{equation*}\begin{aligned}&\textrm{esssup}_{t\in M_{j}}|f(t)-l(t)|
 \\&=\textrm{esssup}_{t\in M_{j}}|f(t)-C_{j}|\\&\leq\frac{1}{2}(G_{j}(f)-g_{j}(f)) \end{aligned}\end{equation*}
Therefore $$||f-l||_{\infty}\leq\frac{1}{2}S_{\mathscr{M}}(f).$$
Thus $$D(f,L_{\mathscr{M}})\leq\frac{1}{2}S_{\mathscr{M}}(f).$$
From Theorem~\ref{main},we get our result.\end{proof}

\begin{cor}Let $(X,\mathscr{A},\mu)$ be an arbitrary probability space, let $\mathscr{M}=\{M_{1},...,M_{k}\}$ be a finite nonempty subset of $\mathscr{A}$ such that  $M_{1},...,M_{k}$ are disjoint  and $\cup_{j=1}^{k}{M_{j}}=X$. Let $$\mathscr{C}=\{(X_{1},...,X_{N})\in\underbrace{X\times...\times X}_{N}|\mathscr{P}=\{X_{1},...,X_{N}\} \textrm{is an} (\mathscr{M},\mu)-\textrm{uniform point set}\}.$$ Then for any $f\in\mathscr{L}_{\infty}(X,\mathscr{A},\mu)$  ,  we have  $$\textrm{esssup}_{(X_{1},...,X_{N})\in\mathscr{C}}|\frac{1}{N}\sum_{n=1}^{N}f(X_{n})-\int_{X}f d\mu|\leq\sum_{j=1}^{k}\mu(M_{j})(G_{j}(f)-g_{j}(f)).$$
\end{cor}
\begin{proof} We extend Niederreiter's proof for Theorem 2 of ~\cite{Nie} to our case.
From the definition of $\mathscr{L}_{\infty}$-norm, we have $$\mu(M_{j})g_{j}(f)\leq\int_{M_{j}}f d\mu\leq\mu(M_{j})G_{j}(f).$$
From the definition of uniform point set, we have
\begin{equation*}\begin{aligned}&\{(X_{1},...,X_{N})\in\mathscr{C}\mid\mu(M_{j})g_{j}(f)>\frac{1}{N}\sum_{\substack{n=1\\X_{n}\in M_{j}}}^{N}f(X_{n})\}
 \\&\subset\cup_{n=1}^{N}(X\times...\times\{X_{n}\in M_{j}\mid g_{j}(f)> f(X_{n})\}...\times X),\end{aligned}\end{equation*} for $1\leq j\leq k$.
\\Thus \begin{equation*}\begin{aligned}&\underbrace{\mu\times...\times\mu}_{N}(\{(X_{1},...,X_{N})\in \mathscr{C}\mid\mu(M_{j})g_{j}(f)>\frac{1}{N}\sum_{\substack{n=1\\X_{n}\in M_{j}}}^{N}f(X_{n})\})
 \\&\leq\sum_{n=1}^{N}\underbrace{\mu\times...\times\mu}_{N}(X\times...\times\{X_{n}\in M_{j}\mid g_{j}(f)> f(X_{n})\}...\times X)\\&=\sum_{n=1}^{N}\mu(\{X_{n}\in M_{j}\mid g_{j}(f)> f(X_{n})\})=0\end{aligned}\end{equation*} for $1\leq j\leq k$, by the definition of $g_{j}(f)$ and Fubini's theorem.
Similarly, $$\underbrace{\mu\times...\times\mu}_{N}(\{(X_{1},...,X_{N})\in\mathscr{C}\mid\mu(M_{j})G_{j}(f)<\frac{1}{N}\sum_{\substack{n=1\\X_{n}\in M_{j}}}^{N}f(X_{n})\})=0.$$Thus $$\textrm{esssup}_{(X_{1},...,X_{N})\in\mathscr{C}}|\frac{1}{N}\sum_{\substack{n=1\\X_{n}\in M_{j}}}^{N}f(X_{n})-\int_{M_{j}}f d\mu|\leq\mu(M_{j})(G_{j}(f)-g_{j}(f)),$$ for $1\leq j\leq k$.
\\Finally, from
\begin{equation*}\begin{aligned}&\frac{1}{N}\sum_{n=1}^{N}f(X_{n})-\int_{X}f d\mu\\&= \sum_{j=1}^{k}(\frac{1}{N}\sum_{\substack{n=1\\X_{n}\in M_{j}}}^{N}f(X_{n})-\int_{M_{j}}f d\mu),\end{aligned}\end{equation*} we get our result.
\end{proof}

\textbf{Acknowledgement:} This work was partially supported by Postdoctoral Science Foundation of China. The authors are grateful to the
anonymous referee for his/her helpful comments.

\end{document}